\documentclass[11pt,reqno]{amsart}

\usepackage{hyperref}
\usepackage{graphicx}
\usepackage{color}
\usepackage{amsfonts,amssymb}
\usepackage{bbm} 

\usepackage{mathrsfs}

\usepackage{pdfsync}

\usepackage{a4wide}

\newtheorem{neu}{}[section]

\newtheorem*{Cor*}{Corollary}
\newtheorem{Thm}[neu]{Theorem}
\newtheorem*{Thm*}{Theorem}
\newtheorem*{Observation*}{Observation}
\newtheorem{Prop}[neu]{Proposition}
\newtheorem*{Prop*}{Proposition}
\theoremstyle{definition}
\newtheorem{Lemma}[neu]{Lemma}
\newtheorem*{Rmk*}{Remark}
\newtheorem{Rmk}[neu]{Remark}
\newtheorem{Ex}[neu]{Example}
\newtheorem*{Ex*}{Example}

\newtheorem*{Qu*}{Question}

\newtheorem{Def}[neu]{Definition}

\newcommand{\N}{\mathbb{N}}

\newcommand{\Z}{\mathbb{Z}}
\newcommand{\R}{\mathbb{R}}
\newcommand{\C}{\mathbb{C}}

\newcommand{\pf}{\longrightarrow}

\newcommand{\hpf}{\hookrightarrow}

\renewcommand{\O}{\mathcal{O}}

\newcommand{\U}{\mathbb{U}}

\newcommand{\E}{\mathbb{E}}

\newcommand{\Crit}{\mathrm{Crit}}

\newcommand{\beq}{\begin{equation}}
\newcommand{\beqn}{\begin{equation}\nonumber}
\newcommand{\eeq}{\end{equation}}

\newcommand{\bea}{\begin{equation}\begin{aligned}}
\newcommand{\bean}{\begin{equation}\begin{aligned}\nonumber}
\newcommand{\eea}{\end{aligned}\end{equation}}

\numberwithin{equation}{section}

\definecolor{Urs}{rgb}{0,.7,0}
\definecolor{Peter}{rgb}{0,0,1}
\definecolor{red}{rgb}{1,0,0}

\newcommand{\He}{\mathscr{H}}

\renewcommand{\>}{\rangle}
\newcommand{\<}{\langle}

\begin{document}
\title{Exponential decay for sc-gradient flow lines}
\author{Peter Albers}
\author{Urs Frauenfelder}
\address{
    Peter Albers\\
    Mathematisches Institut\\
    Westf\"alische Wilhelms-Universit\"at M\"unster}
\email{peter.albers@wwu.de}
\address{
    Urs Frauenfelder\\
    Department of Mathematics and Research Institute of Mathematics\\
    Seoul National University}
\email{frauenf@snu.ac.kr}
\thanks{This material is supported by the SFB 878 Ð Groups, Geometry and Actions (PA) and by the Alexander von Humboldt Foundation (UF)}
\begin{abstract}
In this paper we introduce the notion of sc-action functionals and their sc-gradient flow lines. Our approach is inspired by Floer's unregularized gradient flow. The main result of this paper is that under a Morse condition sc-gradient flow lines have uniform exponential decay towards critical points. The ultimate goal for the future is to construct a M-polyfold bundle over a M-polyfold such that the space of broken sc-gradient flow lines is the zero set of a appropriate sc-section. Here uniform exponential decay is essential. 

Of independent interest is that we derive exponential decay estimates using interpolation inequalities as opposed to Sobolev inequalities. An advantage is that interpolation inequalities are independent of the dimension of the source space.
\end{abstract}
\maketitle
\tableofcontents

\section{Introduction}

In Morse homology one considers a chain complex generated by the critical points of a Morse function and with a boundary operator defined by counting isolated gradient flow lines between critical points. In certain favorable circumstances this recipe even works on infinite dimensional Hilbert manifolds, e.g.~counting closed geodesics on a Riemannian manifold where the Hilbert manifold in question is the loop space of the manifold. A generalization of this situation are Hamiltonian dynamical systems which admit a Lagrangian formulation, see for instance \cite{Abbo_Majer_A_Morse_complex_for_infinite_dimensional_mfds_I} for a modern treatment. Sometimes Hilbert space methods work also in the Hamiltonian formulation, see for instance the celebrated proof by Conley-Zehnder \cite{Conley_Zehnder_The_Birkhoff_Lewis_fixed_point_theorem_and_a_conjecture_of_Arnolf} of the Arnold conjecture on the torus. So far generalizing the beautiful arguments by Conley-Zehnder to general symplectic manifolds failed. Instead, the breakthrough came with Floer's idea of using the unregularized gradient flow equation, see \cite{Floer_The_unregularized_gradient_flow_of_the_symplectic_action}.

Recently Hofer-Wysocki-Zehnder \cite{HWZ_A_General_Fredholm_Theory_I,HWZ_A_General_Fredholm_Theory_II,HWZ_A_General_Fredholm_Theory_III} discovered a new notion of smoothness in infinite dimensions called scale-smoothness which still features the chain rule and tangent functors. Thus, it gives rise to new manifold structures in infinite dimensions such as sc-manifolds and more generally M-polyfolds. In this paper we introduce the notion of sc-action functionals and their sc-gradient flow lines. Our approach is inspired by Floer's unregularized gradient flow. The main result of this paper is that under a Morse condition sc-gradient flow lines have uniform exponential decay towards critical points. The ultimate goal for the future is to construct a M-polyfold bundle over a M-polyfold such that the space of broken sc-gradient flow lines is the zero set of a appropriate sc-section. Here uniform exponential decay is essential. The main result is Theorem  \ref{thm:exponential_decay_full_version} below.

Folk knowledge asserts that near Morse critical points gradient flow lines decay exponentially fast. One of the purposes of this article is to present a framework in which the previous statement is true. For this the sc-language by Hofer-Wysocki-Zehnder \cite{HWZ_A_General_Fredholm_Theory_I,HWZ_A_General_Fredholm_Theory_II,HWZ_A_General_Fredholm_Theory_III} is essential, indeed it enables us to define a good notion of a Morse critical point. Moreover, sc-spaces appeared (in different form) in interpolation theory, see e.g.~\cite{Bergh_Lofstrom_Interpolation_spaces,Triebel_Interpolation_theory_function_spaces_differential_operators}. We make use of interpolation inequalities to prove exponential decay. In particular, the proof does not use special features of the underlying gradient flow equation. Thus, this approach generalizes the statement but sharply contrast the method of proof in previous works on exponential decay in Morse and Floer theory such as \cite{Schwarz_Morse_homology,Salamon_lectures_on_floer_homology,Robbin_Salamon_Asymptotic_behaviour_of_holomorphic_strips,Ziltener_The_invariant_symplectic_action_and_decay_for_vorticies}

\section{Sc-action functionals and sc-gradient flow lines}

In this paper we work with sc-spaces as introduced by Hofer-Wysocki-Zehnder in \cite{HWZ_A_General_Fredholm_Theory_I}. Let 
\beq
\E=(E_k)_{k\in\N}
\eeq
be a sc-space where  each $E_k$ is a Hilbert space. We denote by $\<\cdot,\cdot\>$ the Hilbert-space inner product on $E_0$ and by $||\cdot||_k$ the norm on $E_k$.

\begin{Rmk} 
$\E$ is called a \textit{sc-Hilbert space}, see \cite[Theorem 3.10]{HWZ_A_General_Fredholm_Theory_I}.
\end{Rmk}
We abbreviate
\beq
\E^m:=(E_{k+m})_{k\in\N}\;.
\eeq
Next we define the notion of a \textit{sc-action funtionals} which is inspired by action functionals treated in Floer theory. For this let $U\subset E_0$ be an open set. $U$ induces a sc-Hilbert structure $\U:=(U_k)_{k\in\N}$ with
\beq
U_k:=U\cap E_k\;.
\eeq

In \cite{HWZ_A_General_Fredholm_Theory_I} Hofer-Wysocki-Zehnder introduce the notion of sc-smoothness on sc-spaces. part of the definition of sc-space is the compactness and density of the inclusion $E_{k+1}\hpf E_k$ which implies the chain rule for sc-smooth maps, see \cite[Theorem 2.16]{HWZ_A_General_Fredholm_Theory_I}. This is absolutely crucial for what follows.

\begin{Def}
A sc-smooth map
\beq
A:\U^1\to\R
\eeq 
is a \textit{sc-action functional} if it admits a \textit{gradient}. That is, a sc-smooth vector field $\nabla A:\U^1\to\E$ satisfying
\beq
\<\nabla A(x),\xi\>=DA(x)\xi\quad\forall x\in U_2,\xi\in E_1\;.
\eeq 
In particular,
\beq\label{eqn:defining_eqn_gradient}
\nabla A (U_{k+1})\subset E_k\;.
\eeq
The set of critical points of $A$ is 
\beq
\Crit A:=\{x\in U_\infty\mid DA(x)=0\}=\{x\in U_\infty\mid \nabla A(x)=0\}\subset U_\infty\;.
\eeq
\end{Def}

\begin{Rmk}
Because of the denseness of the embedding $U_2\hpf U_1$ the gradient, if it exists, is uniquely defined. On the other hand not every sc-smooth function admits a gradient, e.g.~the map $x\mapsto \<x,y\>$ where $y\in E_0\setminus E_1$ is fixed, is sc-smooth but does not admit a gradient. In fact, it is even $C^\infty$ in the classical sense when restricted to any level $E_k$.

By definition the critical points are smooth points. However, there might be non-smooth $x\in U\setminus U_\infty$ with $DA(x)=0$, see Section \ref{sec:example}. In Floer theory the critical point equation is elliptic and therefore critical points are automatically smooth. 
\end{Rmk}

\begin{Def}
An sc-action functional $A$ is called \textit{Morse} if the \textit{Hessian} 
\beq
\He_A=D\nabla A:U_{k+2}\times E_{k+1}\to E_k\;,
\eeq 
is for $k\in\N$ an isomorphism at critical points $x\in\Crit A$, i.e.
\beq
\nabla A(x)=0 \quad\Longrightarrow\quad \He_A(x):E_{k+1}\stackrel{\cong}{\pf}E_k\quad\forall k\geq0\;.
\eeq
\end{Def}

\begin{Rmk}\label{rmk:automatically_fractal}
The existence of a Morse critical point immediately implies that the sc-Hilbert space $\E$ is fractal, see \cite{Frauenfelder_First_steps_in_the_geography_of_scale_Hilbert_structures} and \cite{Frauenfelder_Fractal_scale_Hilbert_spaces_and_scale_Hessian_operators}, i.e.~sc-isomorphic to 
\beq
(\ell^2\supset\ell^2_f\subset\ell^2_{f^2}\supset\cdots)
\eeq
for some unbounded monotone increasing function $f:\N\to\R_{>0}$. Here $\ell_f^2\equiv\ell^2_f(\N)$ is the space
\beq
\Big\{(a_n)\in\ell^2(\N)\mid \sum_{n=1}^\infty f(n)a_n^2<\infty\Big\}\;.
\eeq
We point out that there exists a common orthogonal basis $(e_\nu)_\nu\in E_\infty$ for all $E_n$, e.g.~the image of the standard basis of $\ell^2$.

A concrete example of a sc-action functional to keep in mind is 
\beq
E_k:=\ell^2_{f^k}(\Z):=\Big\{(a_n)\in\ell^2(\Z)\mid \sum_{n=-\infty}^\infty (f(n))^ka_n^2<\infty\Big\}
\eeq
with $f(n):=n^2+1$. Then $\displaystyle A(a_n):=\sum_{n=-\infty}^\infty na_n^2$ is a sc-action functional. Indeed the gradient 
\beq
\nabla A(a_n)=\sum_{n=-\infty}^\infty 2na_ne_n
\eeq
is a sc-vector field. Moreover,
\beq
\Crit A=\{(a_n)\mid a_n=0\text{ if }n\neq0\}\;.
\eeq
If we identify $H^1(S^1,\C)\equiv\ell^2_f(\Z)$ via Fourier series the sc-action functional $A$ corresponds to the symplectic area functional. It is important to note that the inner product is defined on $E_0$ whereas the sc-action functional make sense only on $E_1$.
\end{Rmk}

\begin{Lemma}
Let $A$ be a Morse sc-action functional. Then critical points of $A$ are isolated in $E_{10}$, in particular also in $E_\infty$.
\end{Lemma}

\begin{proof}
We give here a nonstandard proof. This  will be useful in the proof of the action-energy inequality in Proposition \ref{prop:action_energy_ineq}. We first define the energy functional $e:\U^1\to\R$ associated with the sc-action functional $A$
\beq\label{eqn:definition_of_energy}
e(x):=\<\nabla A(x),\nabla A(x)\>
\eeq
which is sc-smooth since the $A$ is sc-smooth and the map $(\xi,\eta)\mapsto\<\xi,\eta\>$ is sc-smooth. We observe that 
\beq
e(x)=0 \quad\text{if and only if}\quad DA(x)=0\;.
\eeq
For simplicity we assume that the critical point is $0\in\Crit A$, then
\beq
D^2e(0)[\xi,\eta]=2\<\He_A(0)\xi,\He_A(0)\eta\>\;.
\eeq
It follows from \cite[Proposition 2.14]{HWZ_A_General_Fredholm_Theory_I} that
\beq
e:U_4=(\U^1)_3\to\R
\eeq
is a $C^3$-function. Thus by Taylor's formula we can write
\bea
e(x)&=\tfrac12D^2e(0)(x,x)+\O(||x||^3_4)\\
&=||\He_A(0)x||_0^2+\O(||x||^3_4)\\
&\geq C(||x||_1^2-||x||_4^3)
\eea
where we used that $\He_A(0)$ is an isomorphism. In view of the interpolation inequality \eqref{eqn:interpolation} we continue
\beq\label{eqn:lower_estimate_on_e}
e(x)\geq C(||x||_1^2-C'||x||_{10}||x||_{1}^2)=C(1-C'||x||_{10})||x||_{1}^2\;.
\eeq
Thus, if $||x||_{10}$ is sufficiently small then $e(x)>0$ and thus $x\not\in\Crit A$. In particular, $0\in\Crit A$ is isolated. This proves the Lemma.
\end{proof}

\begin{Prop}[Action-Energy Inequality]\label{prop:action_energy_ineq}
We assume that $x^+\in\Crit A$ with $A(x^+)=a^+$. Then there exists $\epsilon=\epsilon(x^+)>0$ and $\kappa=\kappa(x^+)>0$ such that
\beq
|A(x)-a^+|\leq \frac{e(x)}{\kappa}\qquad\forall ||x-x^+||_{10}<\epsilon\;.
\eeq 
\end{Prop}

\begin{Rmk}
If $A$ is the symplectic area functional the action-energy inequality is a consequence of the isoperimetric inequality. We use the action-energy inequality to prove exponential decay for sc-gradient flow lines. This approach has been used before in \cite{Gaio_Salamon_GW_invariants_of_symplectic_quotients_and_adiabatic_limits} and \cite{Ziltener_The_invariant_symplectic_action_and_decay_for_vorticies}.
\end{Rmk}

\begin{proof}
Without loss of generality we assume that $x^+=0\in\Crit A$ with $A(0)=0$. \cite[Proposition 2.14]{HWZ_A_General_Fredholm_Theory_I} implies that $A:U_4\to\R$ is $C^3$ and thus
\beq
|A(x)|\leq|\<\He_A(0)x,x\>| + C_1||x||_4^3
\eeq
for sufficiently small $x$. Since $\He_A(0)$ is bounded we can estimate further using \eqref{eqn:interpolation}
\beq
|A(x)|\leq C_2(1+||x||_{10})||x||_1^2
\eeq
and employing inequality \eqref{eqn:lower_estimate_on_e} 
\beq
|A(x)|\leq \frac{C_2(1+||x||_{10})}{C_3(1-C_4||x||_{10})}e(x)\;.
\eeq
Thus, for $||x||_{10}$ sufficiently small the Proposition follows.
\end{proof}

\begin{Def}
A (negative) \textit{sc-gradient flow line} of $A$ is a sc-smooth map $x:\R\to\U$ satisfying
\beq
x'(s)=-\nabla A\big(x(s)\big)
\eeq 
where
\beq
x'(s):=Dx(s)1\;.
\eeq
\end{Def}

\begin{Rmk}
Since $\R$ carries only the constant scale structure the requirement of being smooth for a sc-gradient flow line $x:\R\to\U$ implies that $x(\R)\subset U_\infty$.
\end{Rmk}

\section{Exponential decay for sc-gradient flow lines} 
 
In this section we assume that $A:\U^1\to\R$ is a Morse sc-action functional. Moreover, let $x:\R\to\E$ be a sc-gradient flow line such that
\beq
\lim_{s\to+\infty}x(s)=x^+\in\Crit A
\eeq 
where convergence is with respect to the $E_\infty$-topology. The main result of this paper is that $E_\infty$-convergence together with the Morse property imply uniform exponential convergence. In the following we abbreviate
\beq
x^{(m)}:=\frac{d^mx}{ds^m}\qquad \forall m\in\N\;.
\eeq 
for a sc-gradient flow line $x:\R\to\E$.
  
\begin{Thm}[Uniform exponential decay]\label{thm:exponential_decay_full_version}
Let $\kappa$ be the constant from Proposition \ref{prop:action_energy_ineq} and $0<\kappa'<\frac{\kappa}{3}$. Then there constants $C_{j,m}=C_{j,m}(\kappa')>0$, $j,m\in\N$, such that
\beq\label{eqn:exponential_decay_for_x}
||x(s)-x^+||_j\leq C_{j,0}e^{-\kappa' s}\quad\forall s\geq 0
\eeq
and
\beq\label{eqn:exponential_decay_for_derivatives_of_x}
||x^{(m)}(s)||_j\leq C_{j,m}e^{-\kappa's}\quad\forall s\geq 0,\;m\geq1\;.
\eeq
\end{Thm} 
 
As a first step we prove the following Proposition.
 
\begin{Prop}\label{prop:exponential_decay}
Let $\kappa$ be the constant from Proposition \ref{prop:action_energy_ineq} and $0<\kappa'<\frac{\kappa}{2}$. Then there constants $C_j=C_j(\kappa')>0$, $j\in\N$, such that
\beq
||x(s)-x^+||_j\leq C_je^{-\kappa's}\quad\forall s\geq 0
\eeq
Furthermore, there exists a constant $C>0$
\beq
A(x(s))-a^+\leq Ce^{-\kappa s}\qquad\forall s\geq 0
\eeq
where $a^+=A(x^+)$.
\end{Prop}

\begin{proof} 
As a first step we claim that 
\beq
||x(s)-x^+||_0\leq C_0 e^{-\frac{\kappa}{2} s}\qquad\forall s\geq 0
\eeq
where $\kappa$ is the constant from Proposition \ref{prop:action_energy_ineq}. For this we fix $s_0\in\R$ such that
\beq
||x(s)-x^+||_{10}<\epsilon^+\qquad\forall s\geq s_0\;,
\eeq
where $\epsilon^+$ is assumed to be smaller than $\epsilon(x^+)$ from Proposition \ref{prop:action_energy_ineq} and such that there is no other critical point of $A$ which has $||\cdot||_{10}$-distance to $x^+$ less than $\epsilon^+$. Next we prove that if there exists $s_1\in\R$ with $x(s_1)=x^+$ then $x(s)=x^+$ for all $s\geq s_1$ (and thus in this case the desired estimate follows trivially.) If the last assertion is wrong then we can find $s_2>s_1$ with $x(s_2)\neq x^+$. On the other hand the action $A$ strictly decreases along non-constant sc-gradient flow lines:
\bea
A(x(s_2))-A(x(s_1))&=\int_{s_1}^{s_2}\frac{d}{ds}A(x(s))ds\\
&=\int_{s_1}^{s_2}\<\nabla A(x(s)),x'(s)\>ds\\
&=-\int_{s_1}^{s_2}||x'(s)||^2_0ds\\
&<0\;.
\eea
In the same way we conclude
\beq
A(x(s))\leq A(x(s_2)) \qquad\forall s\geq s_2\;.
\eeq
In particular, we arrive at the contradiction
\beq
A(x^+)\leq A(x(s_2))<A(x(s_1))=A(x^+)\;.
\eeq
Therefore, from now on we assume without loss of generality that 
\beq
x(s)\neq x^+\quad\forall s\geq s_0\;.
\eeq
Since by assumption there are no other critical points in a $\epsilon^+$-neighborhood of $x^+$ (measured with the $||\cdot||_{10}$-norm) and since $x(\R)\subset E_\infty$ we conclude that
\beq
\nabla A(x(s))\neq0\qquad\forall s\geq s_0\;.
\eeq
Next we prove the following estimate in this situation:
\beq
||x^+-x(s)||_0\leq \frac{2}{\sqrt{\kappa}}\sqrt{A(x(s))-a^+}
\eeq
for all $s\geq s_0$. Indeed:
\bea
||x^+-x(s)||_0&\leq\int_s^\infty ||x'(t)||_0dt\\
&=\int_s^\infty ||\nabla A(x(t))||_0dt\\
&=\int_s^\infty \frac{||\nabla A(x(t))||_0^2}{||\nabla A(x(t))||_0}dt\\
&=\int_s^\infty \frac{||\nabla A(x(t))||_0^2}{\sqrt{e(x(t))}}dt\\
&\leq\frac{1}{\sqrt{\kappa}}\int_s^\infty \frac{||\nabla A(x(t))||_0^2}{\sqrt{A(x(t))-a^+}}dt\\
&=-\frac{1}{\sqrt{\kappa}}\int_s^\infty \frac{\<\nabla A(x(t)),x'(t)\>}{\sqrt{A(x(t))-a^+}}dt\\
&=-\frac{1}{\sqrt{\kappa}}\int_s^\infty \frac{DA(x(t))x'(t)}{\sqrt{A(x(t))-a^+}}dt\\
&=-\frac{1}{\sqrt{\kappa}}\int_s^\infty \frac{\frac{d}{dt}A(x(t))}{\sqrt{A(x(t))-a^+}}dt\\
&=-\frac{1}{\sqrt{\kappa}}\int_s^\infty \frac{\frac{d}{dt}\big(A(x(t))-a^+\big)}{\sqrt{A(x(t))-a^+}}dt\\
&=-\frac{2}{\sqrt{\kappa}}\int_s^\infty \frac{d}{dt}\sqrt{A(x(t))-a^+}dt\\
&=-\frac{2}{\sqrt{\kappa}} \Big(\sqrt{A(x^+)-a^+}-\sqrt{A(x(s))-a^+}\Big)\\
&=\frac{2}{\sqrt{\kappa}}\sqrt{A(x(s))-a^+}\\
\eea
according to the definition of the energy, see \eqref{eqn:definition_of_energy}, and Proposition \ref{prop:action_energy_ineq}. Next we show that $A(x(s))$ converges exponentially fast to $a^+$. For this observe
\bea
\frac{d}{ds}A(x(s))&=DA(x(s))x'(s)\\
&=\<\nabla A(x(s)),x'(s)\>\\
&=-\<\nabla A(x(s)),\nabla A(x(s))\>\\
&=-e(x(s))
\eea
and thus, since $|A(x(s))-a^+|=A(x(s))-a^+$, we get from Proposition \ref{prop:action_energy_ineq}
\bea
\frac{d}{ds}\Big(A(x(s))-a^+\Big)&=-e(x(s))\\
&\leq-\kappa (A(x(s))-a^+)\;.
\eea
This implies that there exists $C>0$
\beq
A(x(s))-a^+\leq Ce^{-\kappa s}\qquad\forall s\geq s_0
\eeq
as claimed in the statement. Combining these estimates we get
\beq
||x^+-x(s)||_0\leq\frac{2}{\sqrt{\kappa}}\sqrt{A(x(s))-a^+}\leq C_0e^{-\frac\kappa2 s}\qquad\forall s\geq 0
\eeq
for some $C_0>0$. It remains to establish the Theorem for the $||\cdot||_j$-norms. A special case of the interpolation inequality \eqref{eqn:general_interpolation_inequality} is
\beq\label{eqn:interpoltation_for_exponential_decay}
||x(s)||_j\leq C_{j,L} ||x(s)||_L^{\frac{j}{L}}||x(s)||_0^{\frac{L-j}{L}}\;.
\eeq
Given $0<\kappa'<\frac{\kappa}{2}$ choosing $L$ sufficiently large and using that $||x(s)||_L$ converges to $0$ (for any $L$) we obtain
\beq
||x(s)||_j\leq C_je^{-\kappa' s}\qquad\forall s\geq0\;.
\eeq
This proves the Proposition.
\end{proof}

To extend Proposition \ref{prop:exponential_decay} to higher derivatives of the sc-gradient flow line we need the following Lemma.

\begin{Lemma}\label{lem:integral_bound_implies_pointwise_bound}
If there exists $\kappa>0$, $T>0$ and $C>0$ such that
\beq
\int_T^{\infty}||x'(s)||^2_0ds\leq Ce^{-\kappa T}\;.
\eeq
and moreover if $||x''(s)||_0<\epsilon$ for all $s\geq T$ then
\beq
||x'(T)||^2_0\leq \sqrt[3]{8\epsilon C e^{-\kappa T}}\;.
\eeq
\end{Lemma}

\begin{proof}
We set
\beq
\mu:=||x'(T)||_0\;.
\eeq
Then for $t\in[T,T+\frac{\mu}{2\epsilon}]$ we estimate
\bea
||x'(t)||_0&\geq||x'(T)||_0-\int_T^{T+\frac{\mu}{2\epsilon}}||x''(s)||_0ds\\
&\geq\frac{\mu}{2}
\eea
and therefore
\bea
Ce^{-\kappa T}&\geq\int_T^{\infty}||x'(s)||^2_0ds\\
&\geq\int_T^{T+\frac{\mu}{2\epsilon}}||x'(s)||^2_0ds\\
&\geq\frac{\mu^3}{8\epsilon}\;.
\eea
Thus,
\beq
||x'(T)||_0\leq\sqrt[3]{8\epsilon C e^{-\kappa T}}
\eeq
as claimed.
\end{proof}

\begin{Rmk}\label{rmk:sc_0_implies_bounded_operator_norm}
Let $T:\U\oplus\E\to\E$ be $sc^0$ and linear in the second entry. Then for every $x\in U_k$, there exists a neighborhood $V_x\subset U_k$ of $x$ and a constant $c_k>0$ such that 
\beq
||T(y,\cdot)||\leq c_k\qquad\forall y\in V_x
\eeq
where $||T(y,\cdot)||$ denotes the $E_k$-$E_k$-operator norm. That is, the operator norm of $T(y,\cdot)$ is locally bounded. Indeed, since $T$ is $sc^0$ there exists $\delta>0$ and a neighborhood $V_x\subset U_k$ such that
\beq
T(V_x\times B_\delta(0))\subset B_1(0)
\eeq
where $B_r(0)$ is the $r$-ball in $E_k$. This implies that $||T(y,\cdot)||<\frac1\delta$ for all $y\in V_x$. In particular, if $K\subset U$ is compact, then the operator norm $||T(x,\cdot)||$, $x\in K$, is uniformly bounded.
\end{Rmk}

Now we prove Theorem \ref{thm:exponential_decay_full_version}.

\begin{proof}[Proof of Theorem \ref{thm:exponential_decay_full_version}]
We already established the assertion for $x(s)$ for all $j$ in Proposition \ref{prop:exponential_decay}. We proceed by a bootstrapping argument. To start the bootstrapping we recall from Proposition \ref{prop:exponential_decay} that there exists a constant $C>0$ with
\beq
A(x(s))-a^+\leq Ce^{-\kappa s}\qquad\forall s\geq s_0
\eeq
where $a^+=A(x^+)$. We estimate
\bea
\int_T^{\infty}||x'(s)||^2ds&=-\int_T^{\infty}DA(x(s))x'(s)ds\\
&=A(x(T))-a^+\\
&\leq Ce^{-\kappa T}
\eea
and thus by Lemma \ref{lem:integral_bound_implies_pointwise_bound} we conclude
\beq
||x'(T)||\leq C'e^{-\frac{\kappa}{3}T}\;.
\eeq
Deriving the equation $x'=-\nabla A(x)$ gives
\beq
x^{(k)}=\sum_{j=1}^{k-1}\;\;\;\sum_{\substack{l_1\geq l_2\geq\ldots\geq l_j\\ \sum_{r=1}^jl_r=k-1}}\;\;C(l_1,\ldots l_j)\;D^{j}\nabla A(x)\big(x^{(l_1)},x^{(l_2)},\ldots,x^{(l_j)}\big)
\eeq
with some combinatorial constants $C(l_1,\ldots l_j)\geq1$. For $x\in E_{j+1}$ the operator norm of $D^{j}\nabla A(x):E_{j}\times\ldots\times E_{j}\to E_0$ is locally bounded, see Remark \ref{rmk:sc_0_implies_bounded_operator_norm}. Since $x(\R)\cup\{x^\pm\}$ is compact the operator norm of $D^{j}\nabla A(x(s))$ is uniformly bounded in $s\in\R$. Thus, by induction uniform exponential decay for $x'$ implies uniform exponential decay for $x^{(k)}$ and therefore the Theorem.

\end{proof}

\section{An example}\label{sec:example}

An easy example to consider is the symplectic area functional perturbed by a smooth Hamiltonian term. The symplectic area functional is a quadratic functional, see Remark \ref{rmk:automatically_fractal}, and the perturbation is very tame namely an $sc^+$-operator. Thus, it is very easy to see that if critical points of the perturbed symplectic action functional are Morse in the usual sense they are sc-Morse and gradient flow lines decay exponentially fast, see for instance \cite{Salamon_lectures_on_floer_homology}.

Therefore, in this section we study a sc-action functional which is not an $sc^+$-perturbation of a quadratic functional. For this we fix $f(n)=n^2:\N_{\geq1}\to\N_{\geq1}$ and consider the sc-space
$\ell^2\supset\ell^2_f\supset\ell^2_{f^2}\cdots$ with
\beq
\ell^2_f:=\Big\{(a_n)_{n\geq 1}\mid a_n\in\R,\;\sum_{n=1}^\infty f(n)a_n^2<\infty\Big\}\;.
\eeq
We set
\bea
A:\ell^2_f&\to\R\\
(a_n)&\mapsto\sum_{n=1}^\infty \left(na_n^2+ n^3a_n^3\right)
\eea
and claim that $A$ is a sc-action functional. The map $(a_n)\mapsto \sum na_n^2$ is a sc-action functional, see Remark \ref{rmk:automatically_fractal}. Since sc-action functionals form a vector space it remains to check that $(a_n)\mapsto B(a):=\sum_{n=1}^\infty n^3a_n^3$ is a sc-action functional where $a=(a_n)$.

In general it holds 
\beq
||(a_n)||_{\ell^p}\leq||(a_n)||_{\ell^q}
\eeq
for $p\geq q$. Thus, we conclude
\beq
\sum_{n=1}^\infty n^3a_n^3=||(na_n)||_{\ell^3}^3\leq||(na_n)||_{\ell^2}^3=||a_n||_{\ell_f^2}^3<\infty\;,
\eeq
in particular, $B:\ell_f^2\to\R$ is well-defined. Moreover, using the same inequality it follows that 
\beq
B|_{\ell_{{f^k}}^2}:\ell_{f^k}^2\to\R
\eeq
is smooth in the classical sense, i.e.~$B$ is $C^\infty$ on every level. Thus, it follows from \cite[Proposition 2.15]{HWZ_A_General_Fredholm_Theory_I} that $B$ is sc-smooth. The candidate for the gradient of $B$ is
\beq
\nabla B(a)=\sum_{n=1}^\infty 3n^{3}a_n^{2}e_n
\eeq
where $\{e_n\}$ is the standard orthonormal basis of $\ell^2$. It remains to check that $\nabla B$ is sc-smooth. The first derivative is given by
\beq
D\nabla B(a)\hat{a}=\sum_{n=1}^\infty 6n^{3}a_n\hat{a}_ne_n
\eeq
which is a well-defined continuous map $D\nabla B:\ell^2_{f^{k+1}}\times\ell^2_{f^k}\to\ell^2_{f^{k-1}}$. Indeed, if $a\in\ell^2_{f^{k+1}}$ i.e.~$\sum_{n=1}^\infty n^{2k+2}a_n^2<\infty$ and $\hat a\in\ell^2_{f^{k}}$ i.e.~$\sum_{n=1}^\infty n^{2k}\hat a_n^2<\infty$ then
\bea
\sum_{n=1}^\infty n^{2k-2} n^{6}a_n^2\hat{a}_n^2&\leq\sum_{n=1}^\infty n^{2k+2}a_n^2n^{2k}\hat{a}_n^2\\
&\leq\left(\sum_{n=1}^\infty n^{2k+2}a_n^2\right)\left(\sum_{n=1}^\infty n^{2k}\hat a_n^2\right)\\
&<\infty
\eea
Similarly it follows that $D^2\nabla B$ is well-defined and continuous. The higher derivatives vanish. Thus, $B$ is sc-smooth.

The gradient of $A$ is given by
\beq
\nabla A(a)=\sum_{n=1}^\infty \left(2na_ne_n+ 3n^{3}a_n^{2}e_n\right)\;.
\eeq
The critical point equation is 
\beq
2a_n+3n^{2}a_n^{2}=0\;,
\eeq
thus
\beq
a_n=0\quad\text{or}\quad a_n=-\frac{2}{3n^2}\;.
\eeq
Therefore, there are uncountably many solutions. However since we require that critical points are smooth the only critical points are those solutions where only finitely many $a_n$ are non-zero. In fact if infinitely many $a_n$ are non-zero then $(a_n)\not\in\ell^2_{f^{2}}$. In particular, there are only countably many critical points.

\begin{Rmk}
We point out that $0\in\Crit A$ is not an isolated critical point in the $\ell_f^2$ topology. In particular, there cannot be a Morse Lemma on $\ell_f^2$. Nevertheless, $A$ is a Morse sc-action functional since at $a=0$ we have
\beq
\He_A(0)(\hat a_n)=\sum_{n=1}^\infty2n\hat a_ne_n\;.
\eeq 
In general, if $a\in\Crit A$ with $a_i=-\frac{2}{3i^2}$ for $i\in I\subset\N$ and $I$ finite then
\beq
\He_A(a)(\hat a_n)=\sum_{n\not\in I}2n\hat a_ne_n-\sum_{n\in I}2n\hat a_ne_n\;.
\eeq 
Thus, the Morse index of $a$ equals $|I|$ in particular all critical points have finite Morse index but infinite Morse coindex.

Similarly, one can replace the index set $\N$ by $\Z$. Then all critical points have infinite Morse index and coindex. But, as in Floer theory, the index differences are finite.
\end{Rmk}

Since the gradient of $A$ is of diagonal form a gradient flow line $x=(x_n):\R\to\ell^2_\infty$ satisfies
\beq
x_n'(s)=-2nx_n(s)-3n^{3}(x_n(s))^{2}\;.
\eeq
This is a special case of a Bernoulli ODE with solution
\beq
x_n(s)=\frac{1}{c_ne^{2ns}-\frac32n^2}
\eeq
where $c_n$ is the integration constant. As to be expected from our discussion of the Morse indices the only finite energy solution solution converges at $+\infty$ to $0$ and at $-\infty$ to $-\tfrac{2}{3n^2}$. It does so exponentially fast uniformly on all levels and with all derivatives.

\appendix
\section{Interpolation inequalities}\label{sec:app}

The following Proposition is a special case of the interpolation theorem by Stein-Weiss.

\begin{Prop}\label{prop:interpolation_ineq}
Let $\E$ be a fractal sc-space, see Remark \ref{rmk:automatically_fractal}. Then for $0\leq i<j<k$
\beq\label{eqn:general_interpolation_inequality}
||x||_j\leq C_{i,j,k} ||x||_k^{\frac{j-i}{k-i}}||x||_i^{\frac{k-j}{k-i}}\quad\forall x\in E_i\;.
\eeq 
\end{Prop}

\begin{proof}
By definition of being fractal we may assume without loss of generality that
\beq
\E=(\ell^2\supset\ell^2_f\subset\ell^2_{f^2}\supset\cdots)
\eeq
for some unbounded monotone increasing function $f:\N\to\R_{>0}$. Here $\ell_f^2$ is the space
\beq
\Big\{(a_n)\in\ell^2\mid \sum_n f(n)a_n^2<\infty\Big\}\;.
\eeq
Then the statement follows by combining Theorem 1.18.5 and Theorem 1.3.3(g) in \cite{Triebel_Interpolation_theory_function_spaces_differential_operators}. In the notation of the book choose $A=\R$, $X=\N$, and $\mu$ the counting measure, i.e.~every point has measure one. Then we can apply Theorem 1.18.5 with $p=p_0=p_1=2$, $w_0^2:=f^k$, $w_1^2:=f^i$, and $\theta:=\frac{k-j}{k-i}$. After that Theorem 1.3.3(g) relates the norm on the space $A_0=L_{2,w_0^2(x)}(\R)=\ell^2_{f^k}$ and $A_1=\ell^2_{f^i}$.
\end{proof}

We now present an elementary proof of Proposition \ref{prop:interpolation_ineq} which relies on the fractal property of the sc-space $\E$. Namely we note that there exists a natural sc-isomorphism (even a sc-isometry)
\beq
\E\cong\E^m
\eeq
 for all $m\geq0$. The proof shows that the constants $C_{i,j,k}=1$.
\begin{proof}[An elementary proof of Proposition \ref{prop:interpolation_ineq}] We proceed in three steps.
\\ \\
\textbf{Step\,1: } \emph{$||x||_1 \leq ||x||_2^{1/2}||x||_0^{1/2}$ for every $x \in \ell^2_{f^2}$.}
\\ \\
We equivalently show
$$||x||_1^4 \leq ||x||_2^2||x||_0^2, \quad \forall\,\,x \in \ell^2_{f^2}.$$
To see that we estimate
\begin{eqnarray*}
||x||_1^4&=&\Big(\sum_{i=1}^\infty f(i)x_i^2\Big)^2\\
&=&\sum_{i=1}^\infty f(i)^2x_i^4+2\sum_{i < j} f(i)f(j)x_i^2x_j^2\\
&\leq&\sum_{i=1}^\infty f(i)^2x_i^4+\sum_{i < j}\big( f(i)^2+f(j)^2\big)x_i^2x_j^2\\
&=&\Big(\sum_{i=1}^\infty f(i)^2 x_i^2\Big)\Big(\sum_{j=1}^\infty x_j^2\Big)\\
&=&||x||_2^2 ||x||_0^2.
\end{eqnarray*}
This finishes the proof of Step\,1.
\\ \\
\textbf{Step\,2: } \emph{We prove the theorem if $i=0$.}
\\ \\
The proof is by a induction on $k$. For $k \geq 2$ denote by $(A)_k$ the assertion for
every $j \in \{1, \ldots, k-1\}$. The assertion $(A)_2$ is already true in view of Step\,1. Hence it suffices to explain the induction step. For that purpose suppose that $(A)_{k-1}$
is true. By induction hypothesis we have
\begin{equation}\label{ind1}
||x||_1 \leq ||x||_{k-1}^{\frac{1}{k-1}}||x||_0^{\frac{k-2}{k-1}}
\end{equation}
as well as
\begin{equation}\label{ind2}
||x||_{k-2} \leq ||x||_{k-1}^{\frac{k-2}{k-1}}||x||_0^{\frac{1}{k-1}}.
\end{equation}
Now observe that $\ell^2_f$ is isometric to
$\ell^2$ via the map $\{x_\nu\} \mapsto \{f(\nu)^{1/2}x_\nu\}$. This maps restricts for every $m \in \mathbb{N}$ to an isometry $\ell^2_{f^{m+1}} \to \ell^2_{f^m}$. Therefore we obtain from
(\ref{ind2}) by shifting the inequality
\begin{equation}\label{ind3}
||x||_{k-1}\leq||x||_k^{\frac{k-2}{k-1}}||x||_1^{\frac{1}{k-1}}.
\end{equation}
Using (\ref{ind1}) and (\ref{ind3}) we estimate
\begin{eqnarray*}
||x||_1 &\leq& ||x||_{k-1}^{\frac{1}{k-1}}||x||_0^{\frac{k-2}{k-1}}\\
&\leq&||x||_k^{\frac{k-2}{(k-1)^2}}||x||_1^{\frac{1}{(k-1)^2}}||x||_0^{\frac{k-2}{k-1}}
\end{eqnarray*}
Dividing both sides of this inequality by $||x||_1^{\frac{1}{(k-1)^2}}$ we obtain
$$||x||_1^{\frac{k(k-2)}{(k-1)^2}}\leq||x||_k^{\frac{k-2}{(k-1)^2}}||x||_0^{\frac{k-2}{k-1}}$$
or equivalently
\begin{equation}\label{ind4}
||x||_1 \leq ||x||_k^{\frac{1}{k}}||x||_0^{\frac{k-1}{k}}.
\end{equation}
This finishes the proof of $(A)_k$ in the case $j=1$. Now suppose that $j \geq 2$. In view of the induction hypothesis we have
$$||x||_{j-1} \leq ||x||_{k-1}^{\frac{j-1}{k-1}}||x||_0^{\frac{k-j}{k-1}}.$$
Shifting this inequality again we obtain
\begin{equation}\label{ind5}
||x||_j \leq ||x||_k^{\frac{j-1}{k-1}}||x||_1^{\frac{k-j}{k-1}}.
\end{equation}
Combining this inequality with (\ref{ind4}) we obtain
\begin{eqnarray*}
||x||_j &\leq& ||x||_k^{\frac{j-1}{k-1}}||x||_1^{\frac{k-j}{k-1}}\\
&\leq& ||x||_k^{\frac{j-1}{k-1}}||x||_k^{\frac{k-j}{(k-1)k}}||x||_0^{\frac{k-j}{k}}\\
&=&||x||_k^{\frac{j}{k}}||x||_0^{\frac{k-j}{k}}.
\end{eqnarray*}
This finishes the proof of the induction step and hence of Step\,2.
\\ \\
\textbf{Step\,3: } \emph{We prove the theorem.}
\\ \\
The theorem follows immediately from Step\,2 by shifting the indices in the inequality.
\end{proof}

The following example was used in the proof of the action-energy inequality.

\begin{Ex}
For $i=1$, $j=4$, and $k=10$ we obtain
\beq
||x||_4\leq C' ||x||_{10}^{\frac{4-1}{10-1}}||x||_1^{\frac{10-4}{10-1}}=C' ||x||_{10}^{\frac13}||x||_1^{\frac23}\quad\forall x\in E_{10}
\eeq 
thus
\beq\label{eqn:interpolation}
||x||_4^3\leq C||x||_{10}||x||_{1}^2\quad\forall x\in E_{10}\;.
\eeq
\end{Ex}

\section{Sobolev vs.~interpolation}\label{app:sobelev_vs_interpolation}

Uniform exponential decay estimates for Morse theory can be found in \cite{Schwarz_Morse_homology}, for Floer cylinders in \cite{Salamon_lectures_on_floer_homology}, and for holomorphic strips in \cite{Robbin_Salamon_Asymptotic_behaviour_of_holomorphic_strips}. These proofs crucially use Sobolev inequalities. In this paper we do not use Sobolev inequalities but interpolation inequalities instead, see Appendix \ref{sec:app}. They have the advantage that the dimension of the underlying spaces do not enter and therefore give rise to a unified approach to exponential decay estimates. 

In \cite{Frauenfelder_First_steps_in_the_geography_of_scale_Hilbert_structures} it is proved that the sc-isomorphism type of the sc-space
\beq
\ell^{f}=(\ell^2\supset\ell^2_f\subset\ell^2_{f^2}\supset\cdots)
\eeq
is uniquely  described by the growth type of $f:\N\to\R_+$. That is, 
\beq
\ell^{f}\cong\ell^g
\eeq
if and only if there exists a constant $c>0$ such that
\beq
\frac{1}{c}f\leq g\leq cf\;.
\eeq
Let $N$ be a closed manifold then Kang proved in \cite[Theorem A]{Kang_Local_invariant_for_scale_structures_on_mapping_spaces} that the sc-space
\beq
L^2(N,\R^d)\supset W^{1,2}(N,\R^d)\supset W^{2,2}(N,\R^d)\supset \ldots
\eeq
is sc-isomorphic to $\ell^f$ with 
\beq
f(n)=n^{\frac{2}{\dim N}}\;.
\eeq
This reminds of the Sobolev inequalities which depend on the dimension of the source space. The advantage of the interpolation inequalities used in this article is that they are independent of the growth type of $f$. In particular, in the case of mapping spaces they are independent of the dimension of the source space. Thus, they lead to a unified approach to exponential decay estimates.

%
%
%
%
\bibliographystyle{amsalpha}
\bibliography{../../../../Bibtex/bibtex_paper_list}
\end{document}